\documentclass[12pt,a4paper]{article}
\usepackage[T1]{fontenc}
\usepackage{amsmath}
\usepackage{amsfonts}
\usepackage{amssymb}
\usepackage{amsthm}
\usepackage{amscd}
\usepackage{graphicx}
\usepackage{hyperref}
\usepackage{enumitem}
\usepackage{authblk}

\theoremstyle{plain}

\newtheorem{thr}{Theorem}[section]

\theoremstyle{definition}
\newtheorem{defi}[thr]{Definition}
\newtheorem{remark}[thr]{Remark}
\newtheorem{examp}[thr]{Example}
\def\bp{\begin{proof}[\textbf{Proof.}]}
\def\ep{\end{proof}}
\def\bo{\begin{proof}[\textbf{Solution.}]}
\def\eo{\end{proof}}
\def\be{\begin{enumerate}}
\def\ee{\end{enumerate}}
\def\bi{\begin{itemize}}
\def\ei{\end{itemize}}

\title{Rectifying curves in the $n$-dimensional Euclidean space}

\author[1]{Stijn CAMBIE}
\author[2]{Wendy GOEMANS\footnote{Correspondence: wendy.goemans@kuleuven.be}}
\author[1]{Iris VAN DEN BUSSCHE}
\affil[1]{Department of Mathematics, Faculty of Science, KU Leuven, Leuven, Belgium}
\affil[2]{Faculty of Economics and Business, KU Leuven, Brussel, Belgium}

\date{}

\begin{document}

\maketitle


\begin{abstract}
In this article, we study the so-called rectifying curves in arbitrary dimensional Euclidean space.
A curve is said to be a rectifying curve if, in all points of the curve, the orthogonal complement of its normal vector contains a fixed point.
If this fixed point is chosen to be the origin, then this condition is equivalent to saying that the position vector of the curve in every point lies in the orthogonal complement of its normal vector.

Here we characterize rectifying curves in the $n$-dimensional Euclidean space in different ways: using conditions on their curvatures, with an expression for the tangential component, the normal component or the binormal components of their position vector, and by constructing them starting from an arclength parameterized curve on the unit hypersphere.
\end{abstract}


\subsection*{Status of this article: published}
This article appeared in Turkish Journal of Mathematics and the reference therefore is 

Stijn Cambie, Wendy Goemans and Iris Van den Bussche, \emph{Rectifying curves in the $n$-dimensional Euclidean space}, Turk J Math, 40, (2016), 210-223 \textcopyright T\"UBITAK 2016. doi:10.3906/mat-1502-77, \href{http://journals.tubitak.gov.tr/math/issue.htm?id=2041}{http://journals.tubitak.gov.tr/math/issue.htm?id=2041}

\section{Introduction}
Let $\mathbb{E}^n$ denote the $n$-dimensional Euclidean space, that is, $\mathbb{R}^n$ equipped with the standard metric $\langle \textbf{v}, \textbf{w} \rangle = \sum _{i=1}^n v_iw_i$ for vectors $\textbf{v}=(v_1, \ldots, v_n), \textbf{w}=(w_1, \ldots, w_n) \in \mathbb{R}^n$.

As can be found in any textbook on elementary differential geometry, for an arclength parameterized space curve $\alpha : I \subset \mathbb{R} \rightarrow \mathbb{E}^3$ from an open interval $I$ of $\mathbb{R}$ to $\mathbb{E}^3$, which has $\alpha '' (s) \neq 0$ in every $s \in I$, one constructs a Frenet frame $ T (s)= \alpha '(s) , N(s) = \frac{T'(s)}{\lVert T' (s) \rVert},$ $B (s) = T (s) \times N(s) $ whose movement along the curve is expressed by the Frenet-Serret equations 
$$ \left\{ \begin{array}{ccccc}
	T'(s) & = & & \kappa (s) N(s), & \\
	N'(s) & = & -\kappa (s) T(s) & & + \tau (s) B(s), \\
	B'(s) & = & & - \tau (s) N(s). & 
\end{array} \right. $$
Here $T(s)$ is the tangent vector, $N(s)$ the normal vector and $B(s)$ the binormal vector.
Since $\alpha$ is arclength parameterized, it has speed $v_\alpha (s) = \lVert \alpha '(s) \rVert = \sqrt{\langle \alpha '(s), \alpha '(s) \rangle} = 1$, hence, the Frenet frame is an orthonormal basis.
The curvature $\kappa (s) := \langle T'(s), N(s) \rangle$ and the torsion $\tau (s):= - \langle B'(s), N(s) \rangle$ determine the curve up to a Euclidean motion of $\mathbb{E}^3$.

Along a space curve, three planes are defined, each time spanned by two of its Frenet vectors: the osculating plane (spanned by $T$ and $N$), the normal plane (spanned by $N$ and $B$) and the rectifying plane (spanned by $T$ and $B$).
It is well-known that a space curve for which all its osculating planes contain a fixed point is a planar curve and vice versa.
Similarly, a space curve is spherical if and only if all its normal planes contain a fixed point.
But, it was only in 2003 that space curves for which all its rectifying planes contain a fixed point were called rectifying curves and studied in depth by B.-Y. Chen in \cite{Chen}.

In \cite{Chen}, several surprising characterizations of rectifying curves were proved.
To name one, a space curve is congruent to a rectifying curve if and only if the ratio of its torsion and curvature is a non-constant linear function of the arclength parameter.
Also, a space curve $\alpha$ is a rectifying curve if and only if, up to parameterization, it is given by $\alpha (t) = a \sec (t+t_0) y(t)$ where $a, t_0 \in \mathbb{R}$ with $a\neq 0$ and $y$ is an arclength parameterized curve on the unit sphere.

Shortly later, in \cite{ChenDillen}, another set of interesting new characterizations and properties of rectifying curves is proved.
More or less simultaneously, but independent, in \cite{IzumiyaTakeuchi}, it is proved that rectifying curves are geodesics on a cone.
Therefore, the authors of that article call rectifying curves conical geodesics.

Thereafter, the concept of a rectifying curve is translated to Minkowski 3-space, where analogous statements can be proved, of course taking into account the causal character of the curve and that of the rectifying plane, see \cite{IlarslanNesovicMink,IlarslanNesovicPetrovic}.

Also, in \cite{IlarslanNesovicE4}, the definition of a rectifying curve is generalized to 4-dimensional Euclidean space and some theorems characterizing these curves are proved.

Meanwhile, one also finds definitions of rectifying curves in other ambient spaces such as, e.g., the three-dimensional sphere \cite{LucasOrtega} and pseudo-Galilean space \cite{OztekinOgrenmis}.

With this article, we want to contribute to the study of rectifying curves and present some results for these curves in $\mathbb{E}^n$.
First, we recall some preliminaries about the theory of curves in $\mathbb{E}^n$. 
Then, we summarize some results about rectifying curves in $\mathbb{E}^4$ from \cite{IlarslanNesovicE4}.
After that, we examine rectifying curves in $\mathbb{E}^n$ and prove some properties and characterizations of these curves.

This article is based on the work that the first and the third author carried out in their undergraduate project under the supervision of the second author.


\section{Preliminaries}

Analogous as for a space curve, for an arclength parameterized curve $\alpha : I \subset \mathbb{R} \rightarrow \mathbb{E}^n$ that is $n$ times continuously differentiable, one can construct a Frenet frame, $T, N, B_1, \ldots ,$ $B_{n-2}$ that satisfies the equations

\begin{equation}\label{FrenetEqEn}\left \{ 
\begin{aligned} 
 T'(s)&= \kappa_1(s) N(s), \\ 
 N'(s)&=-\kappa_1(s) T(s)+ \kappa_2(s) B_1(s), \\
 B_1'(s)&= -\kappa_2(s) N(s) + \kappa_3(s) B_2(s), \\ 
B_i'(s)&=-\kappa_{i+1}(s)B_{i-1}(s)+\kappa_{i+2}(s)B_{i+1}(s)& \text{ with } i \in \{ 2, 3, \ldots , n-3\}, \\
 B'_{n-2}(s)&=-\kappa_{n-1}(s)B_{n-3}(s).
\end{aligned}
\right. \end{equation}
If the curve $\alpha$ is not arclength parameterized, then the right-hand sides of the Equations~\eqref{FrenetEqEn} must be multiplied by the speed $v$ of $\alpha$.

The functions $\kappa _i$ for $i \in \{ 1,2, \ldots , n-1 \}$ are the curvatures of the curve.
All $\kappa_i$ are positive for $i \in \{ 1,2, \ldots , n-2 \}$.

From the proof of the Equations \eqref{FrenetEqEn}, it follows that $\kappa _{n-1} \equiv 0$ if and only if the curve lies in a hyperplane.
This is equivalent to saying that $B_{n-2}$ is a constant vector, which is then perpendicular to that hyperplane.
See for instance \cite{Gluck,Kuhnel}.

Thus, if in every point the position vector of a curve lies in the orthogonal complement of $B_{n-2}$, then that curve lies in a hyperplane and vice versa.

Similarly, if in every point the position vector of an arclength parameterized curve $\alpha$ lies in the orthogonal complement of the tangent vector $T$, then the curve $\alpha$ lies on a hypersphere.
Indeed, we see that the derivative of $\langle \alpha, \alpha \rangle$ is zero, hence $\langle \alpha, \alpha \rangle$ is a constant and thus, $\alpha$ lies on a hypersphere.
Also here the converse is true.

Following this reasoning and inspired by \cite{IlarslanNesovicE4}, we study curves for which in every point the position vector of the curve lies in the orthogonal complement of the normal vector $N$.

In order for the definition to be independent of the coordinates, we state it as follows.

\begin{defi} \label{Def}
A curve $\alpha : I \rightarrow \mathbb{E}^n$ is a \emph{rectifying curve} if for all $s\in I$ the orthogonal complement of $N(s)$ contains a fixed point.
\end{defi}

Possibly after applying a Euclidean motion of $\mathbb{E}^n$, we assume that the fixed point in Definition \ref{Def} is the origin.
Henceforth, since the orthogonal complement of $N(s)$ is $N(s)^\perp := \{ \textbf{v} \in T_{\alpha (s)}\mathbb{E}^n \,|\, \langle \textbf{v}, N(s) \rangle = 0 \}$, the position vector of a rectifying curve $\alpha$ in $\mathbb{E}^n$ can be written as follows, 
\begin{equation} \alpha (s)= \lambda(s) T(s) + \mu_1(s)B_1(s)+ \cdots + \mu_{n-2}(s)B_{n-2}(s), \label{DefRectCurv} \end{equation}
with $\lambda$, $\mu_1$, \ldots , $\mu_{n-2}$ real functions.

In the rest of this article, we assume that all the curvatures of the curves we consider are not identically zero.


\section{Rectifying curves in $\mathbb{E}^4$}

We recall the results obtained about rectifying curves in $\mathbb{E}^4$ from \cite{IlarslanNesovicE4}.

A rectifying curve in $\mathbb{E}^4$ is characterized by its curvatures in the following theorem.

\begin{thr}[\cite{IlarslanNesovicE4}] \label{TheoremCurvaturesE4}
Let $\alpha $ be an arclength parameterized curve in $\mathbb{E}^4$ with non-zero curvatures. Then, $\alpha$ is congruent to a rectifying curve if and only if
\begin{equation} \dfrac{(s+c)\kappa_1(s)\kappa_3(s)}{\kappa_2(s)}+\left( \frac{1}{\kappa_3(s)} \left( \dfrac{(s+c)\kappa_1(s)}{\kappa_2(s)} \right)' \right)'=0 \label{CurvaturesE4} \end{equation}
for some $c\in \mathbb{R}$.
\end{thr}

An analogous statement in $\mathbb{E}^3$ is that an arclength parameterized curve with non-zero curvature $\kappa$ and non-zero torsion $\tau$ is congruent to a rectifying curve if and only if $$ \left( \frac{(s+c)\kappa (s)}{\tau(s)} \right)' = 0. $$
This is of course equivalent to saying that the ratio of the torsion and the curvature of the curve is a non-constant linear function of its arclength, which is shown in \cite{Chen}.
We provide a proof of an analogous statement in $\mathbb{E}^n$ in the next section.

The following theorem is an immediate consequence of Theorem \ref{TheoremCurvaturesE4}.

\begin{thr}[\cite{IlarslanNesovicE4}]\label{TheoremCstCurvaturesE4}
There exists no rectifying curve in $\mathbb{E}^4$ with non-zero constant curvatures.
\end{thr}

If two of its curvatures are assumed to be a non-zero constant, then the third curvature of a rectifying curve is completely determined by using Theorem \ref{TheoremCurvaturesE4}.
That is, we have the following theorem, in which we correct the last two statements compared to \cite{IlarslanNesovicE4}.

\begin{thr}\label{ThCstCurvE4}
Let $\alpha$ be an arclength parameterized curve in $\mathbb{E}^4$ with non-zero curvatures. When two of its curvatures are assumed to be constant, then $\alpha$ is congruent to a rectifying curve if and only if either,
\begin{enumerate}[label=(\roman{*})]
	\item $\kappa_1(s)=\kappa_1,\kappa_2(s)=\kappa_2$ and $\kappa_3(s)= \pm  \dfrac{1}{ \sqrt{-s^2-2cs + c_1}}$ for some $ \kappa_1, \kappa _2 \in \mathbb{R}^+_0 $ and $c,c_1 \in \mathbb{R}$ while $-s^2-2cs + c_1>0$.
	\item $\kappa_2(s)=\kappa_2,\kappa_3(s)=\kappa_3$ and $\kappa_1(s)= c_1 \dfrac{  \sin(\kappa_3 s +c_2)  }{s+c}$ for some $ \kappa _2 \in \mathbb{R}^+_0 $, $\kappa _3 \in \mathbb{R}_0 $ and $c,c_2 \in \mathbb{R}$, $c_1 \in \mathbb{R}_0 $.
	\item $\kappa_1(s)=\kappa_1,\kappa_3(s)=\kappa_3$ and $\kappa_2(s)= c_2  (s+c)  \sec(\kappa_3 s + c_1)$ for some $ \kappa_1 \in \mathbb{R}^+_0 $, $ \kappa_3 \in \mathbb{R}_0 $ and $c,c_1 \in \mathbb{R}$, $c_2 \in \mathbb{R}_0 $.
\end{enumerate}
\end{thr}
\begin{proof}
If $\alpha$ is a rectifying curve, then, in each case, the statement follows from solving the differential equation that results from Equation \eqref{CurvaturesE4}.
Indeed, if for instance $\kappa_2$ and $\kappa_3$ are constants different from zero, Equation \eqref{CurvaturesE4} reduces to $$\kappa_3^2(s+c)\kappa_1(s) +\big((s+c) \kappa_1(s) \big)''=0.$$
Putting $Y=(s+c)\kappa_1(s)$, this second order linear differential equation is equivalent to $$ \kappa_3^2 Y + Y''=0$$ which has the non-trivial solutions $Y= A\sin(\kappa_3 s) + B\cos(\kappa_3 s)$ with $A,B\in\mathbb{R}$ and $A^2+B^2\neq 0$.
Therefore, $$\kappa_1(s)= c_1\frac{ \sin(\kappa_3 s +c_2)}{s+c}\text{ with } c,c_2\in\mathbb{R}, c_1 \in \mathbb{R}_0 .$$
Conversely, inserting the curvature conditions in Equation \eqref{CurvaturesE4}, it immediately follows from Theorem \ref{TheoremCurvaturesE4} that $\alpha$ is a rectifying curve in every case.
\end{proof}

Rectifying curves in $\mathbb{E}^4$ can also be characterized by the tangential, by the normal or by the first and the second binormal component of their position vector as is illustrated by the following theorem.

\begin{thr}[\cite{IlarslanNesovicE4}]
Let $\alpha$ be an arclength parameterized rectifying curve in $\mathbb{E}^4$ with non-zero curvatures. Then the following statements hold.
\begin{enumerate}[label=(\roman{*})]
\item The distance function $\rho(s)=\lVert \alpha(s) \rVert $ satisfies $\rho^2(s)=s^2+c_1 s + c_2$ for some $c_1,c_2 \in \mathbb{R}$.
\item The tangential component of the position vector of the curve is given by $\langle \alpha(s),T(s) \rangle $ $ = s+c$ for some constant $c \in \mathbb{R}$.
\item The normal component of the position vector of the curve, which is given by $\alpha^N(s)=\langle \alpha(s),N(s)\rangle N(s)+\langle \alpha(s),B_1(s)\rangle B_1(s) +\langle \alpha(s),B_2(s)\rangle B_2(s)$, has constant length and the distance function $\rho(s)$ is non-constant.
\item The first binormal component and the second binormal component of the position vector of the curve are respectively given by
\begin{align*}
& \langle  \alpha(s),B_1(s) \rangle =\frac{\kappa_1(s) (s+c) }{k_2(s)},\\
& \langle \alpha(s),B_2(s) \rangle = \frac{1}{\kappa_3(s)} \left( \dfrac{(s+c)\kappa_1(s)}{\kappa_2(s)} \right)',    
\end{align*}
for some $c\in\mathbb{R}$.
\end{enumerate}
Conversely, if $\alpha$ is an arclength parameterized curve in $\mathbb{E}^4$ with non-zero curvatures and one of these statements holds, then $\alpha$ is a rectifying curve.
\end{thr}
In the next section we generalize this theorem to rectifying curves in $\mathbb{E}^n$.

Finally, to construct rectifying curves in $\mathbb{E}^4$, one can, analogous to $\mathbb{E}^3$, start from an arclength parameterized curve on the unit hypersphere.

\begin{thr}[\cite{IlarslanNesovicE4}]
Let $\alpha$ be a curve in $\mathbb{E}^4$ given by $\alpha(t)=\rho(t)y(t)$, where $\rho(t)$ is an arbitrary positive function and $y(t)$ is an arclength parameterized curve in the unit sphere $\mathbb{S}^3(1)$. Then $\alpha$ is a rectifying curve if and only if $$\rho(t) = \frac{a}{\cos(t+t_0)} \text{ with } a \in \mathbb{R}_0 \text{ and } t_0 \in \mathbb{R}. $$ 
\end{thr}
We prove a similar statement for rectifying curves in $\mathbb{E}^n$ in the next section.

\section{Rectifying curves in $\mathbb{E}^n$}

In this section, we generalize some of the known results to rectifying curves in $\mathbb{E}^n$.
However, it is not possible to make this generalization explicit for all characteristics of rectifying curves because a curve in $\mathbb{E}^n$ has in general $n-1$ curvatures.

\subsection{The curvatures of a rectifying curve}
Let $\alpha$ be an arclength parameterized rectifying curve in $\mathbb{E}^n$.
Take the derivative of the position vector of $\alpha$ given by Equation \eqref{DefRectCurv}, that is,
$$ \alpha'(s) = \lambda'(s) T(s) + \lambda(s) T'(s) + \sum_{i=1}^{n-2} \left(\mu'_i(s)B_i(s) + \mu_i(s)B'_i(s)\right). $$
Then, use the Equations \eqref{FrenetEqEn} in this expression, this results in
\begin{multline*}
T(s) = \lambda'(s) T(s) +\big(\kappa_1(s)\lambda(s)  -\kappa_2(s) \mu_1(s)\big)N(s)+\mu_1(s)\kappa_3(s)  B_2(s)\\
+ \sum_{i=1}^{n-2} \mu'_i(s)B_i(s) +\sum_{i=2}^{n-3} \mu_i(s)\big(-\kappa_{i+1}(s)B_{i-1}(s)+\kappa_{i+2}(s)B_{i+1}(s)  \big) \\ - \kappa_{n-1}(s) \mu_{n-2}(s)B_{n-3}(s).
\end{multline*}
Equivalently,
\begin{multline*}
T(s) = \lambda'(s) T(s) +\big(\kappa_1(s)\lambda(s)  -\kappa_2(s) \mu_1(s) \big)  N(s) + \big(\mu'_1(s) -\mu_{2}(s)\kappa_{3}(s) \big)  B_1(s) \\
+ \sum_{i=2}^{n-3} \big(\mu_{i-1}(s)\kappa_{i+1}(s)+ \mu'_i(s) -\mu_{i+1}(s)\kappa_{i+2}(s) \big)  B_i(s) \\
+ \big(\mu_{n-3}(s) \kappa_{n-1}(s)+ \mu'_{n-2}(s) \big) B_{n-2}(s).
\end{multline*}
Since the Frenet frame is an orthonormal basis, this equation leads to 
\begin{subequations}
\label{Fctns}
\begin{align}
&\lambda'(s)=1, \label{Fctnsa} \\
&\lambda(s)\kappa_1(s)  - \mu_1(s)\kappa_2(s) =0, \label{Fctnsb} \\
& \mu'_1(s) -\mu_{2}(s)\kappa_{3}(s) =0, \label{Fctnsc}  \\
&\mu_{i-1}(s)\kappa_{i+1}(s)+ \mu'_i(s) -\mu_{i+1}(s)\kappa_{i+2}(s)   = 0 \text{ with } i \in \{ 2,3, \ldots , n-3 \}, \label{Fctnsd}\\
&\mu_{n-3}(s) \kappa_{n-1}(s)+ \mu'_{n-2}(s) =0. \label{Fctnse}
\end{align}
\end{subequations}

The system \eqref{Fctns} consists of $n$ equations incorporating $n-1$ curvature functions, the function $\lambda$ and $n-2$ functions $\mu_i$ with $i \in \{ 1,2, \ldots , n-2 \}$.
However, these functions $\mu_i$ can be expressed in terms of the curvature functions, derivatives of the curvature functions and the function $\lambda$.
Indeed, from Equation \eqref{Fctnsa} one has $\lambda (s) = s + c $ with $c \in \mathbb{R}$.
Equations \eqref{Fctnsb} and \eqref{Fctnsc} lead to 
$$ \mu _1 (s) = \lambda (s) \frac{\kappa _1 (s)}{\kappa _2 (s)} \quad \mbox{and} \quad \mu_2 (s) = \frac{1}{\kappa _3 (s)} \frac{\kappa _1 (s)}{\kappa _2 (s)} + \frac{\lambda(s)}{\kappa _3 (s)} \left(\frac{\kappa _1 (s)}{\kappa _2 (s)}\right)'. $$
Introducing functions $\mu_{1,0}$, $\mu_{2,0}$ and $\mu_{2,1}$ one rewrites these equations as, 
\begin{equation} \label{mu1andmu2}
\mu _1 (s) = \mu_{1,0} (s) \frac{\kappa _1 (s)}{\kappa _2 (s)} \quad \mbox{and} \quad \mu_2 (s) = \mu_{2,0} (s) \frac{\kappa _1 (s)}{\kappa _2 (s)} + \mu_{2,1} (s) \left(\frac{\kappa _1 (s)}{\kappa _2 (s)}\right)'.
\end{equation}
By induction, from Equations \eqref{Fctnsd}, one finds this way,
\begin{equation} \label{mui} 
\mu_{i} (s) = \sum_{k=0}^{i-1} \mu_{i,k}(s)\frac{\partial ^k}{\partial s^k} \left(\frac{\kappa_1(s)}{\kappa_2(s)}\right) 
\end{equation}
for $i \in \{3,4, \ldots , n-2\}$.
Here the functions $\mu_{i,k}$ are inductively defined by the following system
\begin{equation} \label{HelpFunctions} \left\{ 
\begin{aligned}
\mu_{1,0} (s) & = s + c \quad \mbox{with } c \in \mathbb{R}, \\
\mu_{2,0} (s) & = \frac{1}{\kappa_3(s)}, \quad \mu_{2,1} (s) = \frac{s+c}{\kappa_3(s)}, \quad \mbox{ and for } i \in \{3,4 \ldots , n-2 \} \mbox{ one has},\\
\mu_{i,0}(s) & = \frac{\kappa _{i}(s) \mu_{i-2,0}(s) + \mu_{i-1,0}'(s)}{\kappa_{i+1}(s)}, \\
\mu_{i,k}(s) & = \frac{\kappa _{i}(s) \mu_{i-2,k}(s) + \mu_{i-1,k}'(s)+\mu_{i-1,k-1}(s)}{\kappa_{i+1}(s)} \quad \mbox{for } k \in \{ 1,2, \ldots , i-3\}, \\
\mu_{i,i-2}(s) & = \frac{\mu_{i-1,i-3}(s) + \mu_{i-1,i-2}'(s)}{\kappa_{i+1}(s)}, \\
\mu_{i,i-1}(s) & = \frac{\mu_{i-1,i-2}(s)}{\kappa_{i+1}(s)}.
\end{aligned} \right.
\end{equation}

Based on the system of equations \eqref{Fctns}, we prove the following theorem, which is a higher dimensional version of Theorem \ref{TheoremCurvaturesE4}.

\begin{thr} \label{TheoremCurvaturesEn}
Let $\alpha$ be an arclength parameterized curve in $\mathbb{E}^n$ with non-zero curvatures.
Then, $\alpha$ is congruent to a rectifying curve if and only if 
\begin{equation}
\label{EqRectCurveTh}
\kappa_{n-1} (s)  \sum_{k=0}^{n-4} \mu_{n-3,k} (s) \frac{\partial ^k}{\partial s^k}\left( \frac{\kappa_1(s)}{\kappa_2(s)}\right) + \sum_{k=0}^{n-3} \left( \mu_{n-2,k} (s) \frac{\partial ^k}{\partial s^k}\left( \frac{\kappa_1(s)}{\kappa_2(s)}\right)\right)' = 0
\end{equation}
with $\mu_{i,k}$ inductively defined by the system \eqref{HelpFunctions}.
\end{thr}
\begin{proof}
If $\alpha$ is a rectifying curve, then inserting Equation \eqref{mui} in Equation \eqref{Fctnse} immediately results in Equation \eqref{EqRectCurveTh}.

Conversely, assume that Equation \eqref{EqRectCurveTh} is satisfied.
Define the curve $\beta(s)=\alpha(s)- \lambda(s) T(s) - \mu_1(s)B_1(s)- \cdots - \mu_{n-2}(s)B_{n-2}(s)$ with the function $\lambda(s)=s+c$ where $c \in \mathbb{R}$ and the functions $\mu_1(s),\ldots , \mu_{n-2}(s)$ as in Equations \eqref{mu1andmu2} and \eqref{mui}.
Since $\beta '(s) = 0$, we conclude that $\alpha $ is congruent to a rectifying curve.
\end{proof}

Applying Theorem \ref{TheoremCurvaturesEn} we can proof the following theorem, which is an analogue of Theorem \ref{TheoremCstCurvaturesE4}.

\begin{thr} \label{TheoremNonExistence}
There exists no rectifying curve in $\mathbb{E}^n$ with non-zero constant curvatures.
\end{thr}
\begin{proof}
Assume there exists a rectifying curve with all its curvatures $\kappa_1$, $\kappa_2$, \ldots, $\kappa_{n-1}$ constant but non-zero.
From Equations (\ref{Fctnsa}), (\ref{Fctnsb}) and (\ref{Fctnsc}) it follows that $$ \lambda(s)=s+c, \quad \mu_1(s) = \frac{\kappa_1}{\kappa_2} (s+c), \quad
\mu_{2}(s)=  \frac{\kappa_1}{\kappa_2  \kappa_{3}} \quad \mbox{with } c \in \mathbb{R}.$$
For $i \in \{2,3, \ldots , n-3\}$ we deduce from Equation (\ref{Fctnsd}) that $$\mu_{i+1}(s) = \frac{  \mu_{i-1}(s)\kappa_{i+1}+ \mu'_i(s) }{\kappa_{i+2} }. $$
By induction we prove that 
\begin{subequations}
\label{CstCurvEq}
\begin{align} 
\mu_{2m-1}(s) & = \frac{\kappa_1\kappa_3 \cdots \kappa_{2m-1}}{\kappa_2\kappa_4 \cdots \kappa_{2m}  } (s+c), \label{CstCurvEqa} \\
\mu_{2m}(s) & = \frac{  \sum_{j=1}^{m} \left(\prod_{i=1}^{j} \kappa_{2i-1} \prod_{i=j+1}^{ m }\kappa_{2i}\right)^2}{\kappa_1\kappa_2\kappa_3 \cdots \kappa_{2m+1}}, \label{CstCurvEqb} 
\end{align}
\end{subequations}
where the index ranges from 1 to $n-2$.
Indeed, it is clear that these equations are valid for $m=1$.
If the equations are valid for $m \in \{1,2,\ldots ,M \}$, then,
\begin{align*}\mu_{2M+1} (s) & = \frac{  \mu_{2M-1}(s)\kappa_{2M+1}+ \mu'_{2M}(s) }{\kappa_{2M+2} } =  \frac{\kappa_1\kappa_3 \cdots \kappa_{2M-1}}{\kappa_2\kappa_4 \cdots \kappa_{2M}  } (s+c) \frac{\kappa_{2M+1}}{\kappa_{2M+2}} \\
&=\frac{\kappa_1\kappa_3 \cdots \kappa_{2M-1}\kappa_{2M+1}}{\kappa_2\kappa_4 \cdots \kappa_{2M}\kappa_{2M+2}  } (s+c).\end{align*}
Which is Equation \eqref{CstCurvEqa} for $m=M+1$.

Also, 
\begin{align*}
\mu_{2M+2} (s) &=  \frac{  \mu_{2M}(s)\kappa_{2M+2}+ \mu'_{2M+1}(s) }{\kappa_{2M+3} }\\
&=\frac{  \sum_{j=1}^{M} \left(\prod_{i=1}^{j} \kappa_{2i-1}  \prod_{i=j+1}^{M }\kappa_{2i}    \right)^2                 }{\kappa_1\kappa_2\kappa_3 \cdots \kappa_{2M+1}}    \frac{\kappa_{2M+2}}{\kappa_{2M+3}} + \frac{\kappa_1\kappa_3 \cdots \kappa_{2M-1}\kappa_{2M+1}}{\kappa_2\kappa_4 \cdots \kappa_{2M}\kappa_{2M+2}\kappa_{2M+3} } \\
&=\frac{  \sum_{j=1}^{M} \left[ \left(\prod_{i=1}^{j} \kappa_{2i-1} \prod_{i=j+1}^{M }\kappa_{2i}    \right)^2     \kappa_{2M+2}^2    \right]        }{\kappa_1\kappa_2\kappa_3 \cdots \kappa_{2M+1}\kappa_{2M+2}\kappa_{2M+3}} + \frac{\left(\kappa_1\kappa_3 \cdots \kappa_{2M-1}\kappa_{2M+1}\right)^2}{\kappa_1\kappa_2\kappa_3 \cdots \kappa_{2M+1}\kappa_{2M+2}\kappa_{2M+3} } \\
&=\frac{  \sum_{j=1}^{M} \left(\prod_{i=1}^{j} \kappa_{2i-1} \prod_{i=j+1}^{M+1 }\kappa_{2i}    \right)^2            }{\kappa_1\kappa_2\kappa_3 \cdots \kappa_{2M+1}\kappa_{2M+2}\kappa_{2M+3}} + \frac{\left(\kappa_1\kappa_3 \cdots \kappa_{2M-1}\kappa_{2M+1}\right)^2}{\kappa_1\kappa_2\kappa_3 \cdots \kappa_{2M+1}\kappa_{2M+2}\kappa_{2M+3} } \\
&=\frac{  \sum_{j=1}^{M+1} \left(\prod_{i=1}^{j} \kappa_{2i-1}  \prod_{i=j+1}^{M+1 }\kappa_{2i}    \right)^2            }{\kappa_1\kappa_2\kappa_3 \cdots \kappa_{2M+1}\kappa_{2M+2}\kappa_{2M+3}}\cdot
\end{align*}
Which proves Equation \eqref{CstCurvEqb} for $m=M+1$. 

For even $n$, with the aid of Equations \eqref{CstCurvEq}, Equation \eqref{Fctnse}, which is equivalent to Equation \eqref{EqRectCurveTh}, reduces to $$\frac{\kappa_1\kappa_3 \cdots \kappa_{n-3} }{\kappa_2\kappa_4 \cdots \kappa_{n-2}} (s+c) \kappa_{n-1} = 0.$$
But, since we assume all curvatures to be non-zero, this leads to a contradiction.

For odd $n$, using Equations \eqref{CstCurvEq}, Equation \eqref{Fctnse} is rewritten as follows 
$$ \frac{\sum_{j=1}^{\frac{n-3}{2}} \left(\prod_{i=1}^{j} \kappa_{2i-1} \prod_{i=j+1}^{\frac{n-3}{2} }\kappa_{2i}\right)^2}{\kappa_1\kappa_2\kappa_3 \cdots \kappa_{n-2}} \kappa_{n-1}+\frac{\kappa_1\kappa_3 \cdots \kappa_{n-2} }{\kappa_2\kappa_4 \cdots \kappa_{n-1}} = 0.$$
This is equivalent to
$$ \frac{\sum_{j=1}^{\frac{n-3}{2}} \left(\prod_{i=1}^{j} \kappa_{2i-1} \prod_{i=j+1}^{\frac{n-1}{2} }\kappa_{2i}\right)^2}{\kappa_1\kappa_2\kappa_3 \cdots \kappa_{n-2} \kappa_{n-1}}+\frac{\big(\kappa_1\kappa_3 \cdots \kappa_{n-2} \big)^2}{\kappa_1\kappa_2\kappa_3 \cdots \kappa_{n-2} \kappa_{n-1}} = 0.$$
Hence, again this leads to a contradiction.
That is, using Theorem \ref{TheoremCurvaturesEn}, we conclude that there exists no rectifying curve with all its curvatures non-zero constants.
\end{proof}

\begin{remark}
As is shown in \cite{Kuhnel}, a curve $\beta$ in $\mathbb{E}^n$ which has all its curvatures constant is parameterized by 
\begin{equation}
\beta (t) = \left( a_1 \sin \left( b_1 t \right), a_1 \cos \left( b_1 t \right), \ldots , a_m \sin \left( b_m t \right), a_m \cos \left( b_m t \right) \right) \label{CurveCstCurvEvenn}
\end{equation} for even $n=2m$ and by 
\begin{equation} \beta (t) = \left( a_1 \sin \left( b_1 t \right), a_1 \cos \left( b_1 t \right), \ldots , a_m \sin \left( b_m t \right), a_m \cos \left( b_m t \right), at \right) \label{CurveCstCurvOddn}
\end{equation}
for odd $n=2m+1$.
Here $a , a_i, b_i \in \mathbb{R}$ and all $b_i$ are distinct numbers for $i \in \{ 1,2, \ldots , m\}$.

From this it is clear that, for even $n$, a rectifying curve in $\mathbb{E}^n$ cannot have all its curvatures constant, since then it would lie on a hypersphere. 
In that case, the position vector of the curve lies in every point in $T(s)^\perp$ and hence in general not in $N(s)^\perp $.

From parameterizations \eqref{CurveCstCurvEvenn} and \eqref{CurveCstCurvOddn} one can also straightforwardly show that a curve with all its curvatures constant is not a rectifying curve since then $\langle \beta (t), N(t) \rangle \not\equiv 0$.
\end{remark}

If all but one of the curvatures of a rectifying curve are assumed to be non-zero constants, one can try to determine that non-constant curvature as in the following theorem.

\begin{thr} \label{ThCstCurvEn}
Let $\alpha$ be an arclength parameterized curve in $\mathbb{E}^n$ with non-zero curvatures.
If the first $n-2$ curvatures of $\alpha$ are non-zero constants $\kappa_1$, $\kappa_2$, \ldots, $\kappa_{n-2}$, then, $\alpha$ is a rectifying curve if and only if
\begin{equation} \label{KappaNmin1CsteKappas} \left\{ 
\begin{aligned}
 \kappa_{n-1}(s) &= \pm   \frac{   1 }{  \sqrt{a s(s+2c) + b }} \mbox{ for even $n$,}\\
       \kappa_{n-1}(s) &= \pm  \frac{s+c}{ \sqrt{a s(s+2c) + b}}   \mbox{ for odd $n$.}
\end{aligned}
\right.
\end{equation}
Here $a$ is a constant depending on the curvatures $\kappa_1$, $\kappa_2$, \ldots, $\kappa_{n-2}$ and $b,c \in \mathbb{R}$.
\end{thr}
\begin{proof}
Assume that $\alpha$ is a rectifying curve for which its first $n-2$ curvatures are non-zero constants.
From the system \eqref{Fctns}, we know 
\begin{equation} \mu_{n-3}(s) \kappa_{n-1}(s) =  -\mu'_{n-2}(s) = -\left(\frac{ \mu_{n-4}(s) \kappa_{n-2} + \mu'_{n-3}(s)}{\kappa_{n-1}(s)} \right)'. \label{CstCurvEqPart} \end{equation}
Now, the Equations \eqref{CstCurvEq} remain valid for the index up to $n-3$.
Therefore, for even $n$, Equation \eqref{CstCurvEqPart} reduces to 
$$ a (s+c) \kappa_{n-1}(s) = \left(\frac{1}{\kappa_{n-1}(s)} \right)',$$
with $a$ a constant depending on the curvatures $\kappa_1$, $\kappa_2$, \ldots, $\kappa_{n-2}$.
This differential equation leads to the solution as in the statement of the theorem.

In case $n$ is odd, from Equation \eqref{CstCurvEqPart}, one obtains the differential equation
$$ a \kappa_{n-1}(s) = \left(\frac{s+c}{\kappa_{n-1}(s)} \right)',$$
again with $a$ a constant determined by the curvatures $\kappa_1$, $\kappa_2$, \ldots, $\kappa_{n-2}$.
This differential equation is equivalent to
$$ \frac{ s+c}{\kappa_{n-1}(s)} \left(\frac{ s+c}{\kappa_{n-1}(s)} \right)'  = a(s+c)$$
from which the solution as in the statement of the theorem follows.

Conversely, assume that $\alpha$ is a curve with its first $n-2$ curvatures constant with $n \geq 4$ and its last curvature as in Equations \eqref{KappaNmin1CsteKappas}.
Then, Equation \eqref{EqRectCurveTh} is $$ \kappa _{n-1} (s) \mu_{n-3,0}(s)+\mu_{n-2,0}'(s)=0. $$
Using the system \eqref{HelpFunctions}, this is equivalent to
$$ \kappa _{n-1}(s) \mu_{n-3,0}(s) + \left( \frac{\kappa _{n-2} \mu_{n-4,0} (s) + \mu_{n-3,0}'(s)}{\kappa _{n-1}(s)} \right)' = 0.$$
But, from Equation \eqref{mui} and the constancy of $\kappa_1$ and $\kappa _2$, this is equivalent to Equation \eqref{CstCurvEqPart}.
Since the curvatures defined in Equations \eqref{KappaNmin1CsteKappas} satisfy Equation \eqref{CstCurvEqPart}, we find that Equation \eqref{EqRectCurveTh} is fulfilled, hence, $\alpha$ is a rectifying curve.
\end{proof}

\begin{remark}
If either $\kappa_1$ or $\kappa_2$ of a rectifying curve $\alpha$ is non-constant and all the other curvatures of $\alpha$ are constant, then, inserting in Equation \eqref{Fctnse} the other equations of the system \eqref{Fctns}, one obtains a homogeneous linear differential equation of order $n-2$ with constant coefficients for the function $\mu_1$.
Therefore, the roots of the associated characteristic polynomial determine the solution for $\mu_1$ and by that also the expression for $\kappa _1$ or $\kappa _2$.
However, it is not possible to write down this solution explicitly for arbitrary $n$.

If one of the curvatures $\kappa_i$ with $i \in \{ 3,4, \ldots ,n-2 \}$ of a rectifying curve $\alpha$ is non-constant and all the other curvatures of $\alpha$ are constant, then Equation \eqref{EqRectCurveTh} is a homogeneous non-linear differential equation for $\kappa_i$ which, even for low dimension, cannot be solved explicitly.

Therefore, due to the existence of $n-1$ curvatures, one cannot explicitly solve all possible cases for the curvature that is non-constant.
\end{remark}

\subsection{The components of the position vector of a rectifying curve}
A rectifying curve in $\mathbb{E}^n$ is characterized by its tangential component, its normal component or its binormal components in the following theorem.
\begin{thr} \label{TheoremComponents}
Let $\alpha$ be an arclength parameterized rectifying curve in $\mathbb{E}^n$ with non-zero curvatures.
Then the following statements hold.
\begin{enumerate}[label=(\roman{*})]
\item The tangential component of the position vector of the curve is given by $\langle \alpha(s),T(s) \rangle $ $= s+c$ for some constant $c \in \mathbb{R}$.
\item The distance function $\rho(s)=\lVert \alpha(s)\rVert$ satisfies $\rho^2(s)=s^2+c_1 s + c_2$ for some $c_1,c_2 \in \mathbb{R}$.
\item The normal component $\alpha^N(s)$ of the position vector of the curve has constant length and the distance function $\rho(s)$ is non-constant.
\item The binormal components of the position vector of the curve, for $i \in \{ 1,2, \ldots , n-2 \}$, are given by $\langle \alpha (s) , B_i (s) \rangle = \mu_i (s)$ where the $\mu_i$'s are defined by Equations \eqref{mu1andmu2} and \eqref{mui}.
\end{enumerate}
Conversely, if $\alpha$ is an arclength parameterized curve in $\mathbb{E}^n$ with non-zero curvatures and one of the statements holds, then $\alpha$ is a rectifying curve.
\end{thr}
\begin{proof}
\begin{enumerate}[label=(\roman{*})]
\item In order to prove (i), assume that $\alpha$ is an arclength parameterized rectifying curve.
Then, from Equation \eqref{DefRectCurv} and \eqref{Fctnsa} we see that $$ \langle \alpha (s), T(s) \rangle = \lambda (s) = s +c. $$

Conversely, if $ \langle \alpha (s), T(s) \rangle = s +c $, then, differentiating with respect to $s$ while keeping Equations \eqref{FrenetEqEn} in mind, leads to $ \langle \alpha (s), N(s) \rangle \kappa_1 (s) = 0 $ from which we conclude that $\alpha$ is a rectifying curve since $\kappa _1 \not\equiv 0$.
\item To prove (ii), if $\alpha$ is an arclength parameterized rectifying curve, one uses the system \eqref{Fctns}.
Multiply \eqref{Fctnsc}, \eqref{Fctnsd} and \eqref{Fctnse} with $\mu_i$ where $i \in \{ 1,2, \ldots , n-2 \}$, respectively,
\begin{align*}
&\mu_1(s) \big(\mu'_1(s) -\mu_{2}(s)\kappa_{3}(s) \big) =0,  \\
& \mu_i (s) \big(\mu_{i-1}(s)\kappa_{i+1}(s)+ \mu'_i(s) -\mu_{i+1}(s)\kappa_{i+2}(s)  \big) = 0 \text{ for } i \in \{ 2,3, \ldots , n-3 \}, \\
&  \mu_{n-2}(s) \big( \mu_{n-3}(s) \kappa_{n-1}(s)+ \mu'_{n-2}(s) \big) =0.
\end{align*}
Adding these equations leads to $\sum_{i=1}^{n-2} \mu_i(s) \mu_i'(s) = 0$.
Hence, $\sum_{i=1}^{n-2} \mu^2_i(s) =a^2$ for $a \in \mathbb{R}_0$.
From Equation \eqref{DefRectCurv} we have $$\rho^2(s) = \langle \alpha(s),\alpha(s) \rangle  = \lambda^2(s)+ \sum_{i=1}^{n-2} \mu^2_i(s) = (s+c)^2+a^2$$ where we also used Equation \eqref{Fctnsa}.

Conversely, differentiating $\rho^2(s) = \langle \alpha(s),\alpha(s) \rangle  = s^2+c_1 s + c_2$ twice with respect to $s$ while inserting Equations \eqref{FrenetEqEn} leads to $ \langle \alpha (s), N(s) \rangle = 0 $. Thus, $\alpha$ is a rectifying curve.
\item Now, to prove (iii), decompose the position vector of a curve $\alpha$ in its tangential and its normal component, that is, $$\alpha(s)= \langle \alpha (s), T(s) \rangle T(s)+ \alpha^N(s).$$

For a rectifying curve, from Equation \eqref{DefRectCurv}, we know that $\alpha^N(s)= \sum _{i=1} ^{n-2} \mu_i(s) B_i(s)$.
Therefore, $ \lVert \alpha^N(s) \rVert = \sqrt{ \sum_{i=2}^{n-2} \mu^2_i(s) } =a$ in which we use what we already calculated in part (ii) above.
Thus, the normal component has constant length.
The claim about the distance function $\rho$ is also already proved in part (ii) above.

Conversely, from $\alpha^N(s) =  \alpha(s) - \langle \alpha (s), T(s) \rangle T(s)$ and $\langle \alpha^N(s) , \alpha^N(s) \rangle = a^2$, we see that 
$$ a^2 = \langle \alpha^N(s) , \alpha^N(s) \rangle = \langle \alpha(s),\alpha(s) \rangle - \langle \alpha(s),T(s) \rangle ^2 .$$
Differentiating with respect to $s$ and using Equations \eqref{FrenetEqEn}, we obtain $$ \kappa_1(s) \langle \alpha(s),T(s) \rangle  \langle \alpha(s),N(s) \rangle = 0. $$ 
Since $\rho^2(s)= \lVert \alpha (s) \rVert ^2$ is not a constant, $\langle \alpha(s),T(s) \rangle$ must be different from zero. Hence, $\alpha $ is a rectifying curve.
\item Finally, (iv) follows immediately from Equation \eqref{DefRectCurv} for the position vector of a rectifying curve.

Conversely, assume $\langle \alpha (s), B_1 (s) \rangle = \mu_1 (s)$.
Take the derivative and use Equations~\eqref{FrenetEqEn}, then, $$ - \kappa _2 (s) \langle \alpha (s), N(s) \rangle + \kappa _3 (s) \mu_2 (s) = \mu_1'(s) $$
where we also inserted $\langle \alpha (s), B_2 (s) \rangle = \mu_2(s)$.
Using the definition of $\mu_1$ and $\mu_2$, one obtains $\langle \alpha (s), N(s) \rangle = 0$.
Therefore, $\alpha$ is a rectifying curve.
\end{enumerate}
\end{proof}
Remark that, to prove the converse of the last statement of Theorem \ref{TheoremComponents}, only the binormal components with respect to the first two binormals $B_1$ and $B_2$ are required.

\subsection{A classification of rectifying curves}
Finally, we construct rectifying curves starting from an arclength parameterized curve on the unit hypersphere centered at the origin, $\mathbb{S}^{n-1} (1)= \{ p \in \mathbb{E}^n \mid \langle p, p \rangle = 1\}$.

\begin{thr} \label{ThConstrRectCurveEn}
Let $\alpha$ be a curve in $\mathbb{E}^n$ given by $\alpha(t)=\rho(t)y(t)$, where $\rho(t)$ is an arbitrary positive function and $y(t)$ is an arclength parameterized curve in the unit hypersphere $\mathbb{S}^{n-1}(1)$.
Then $\alpha$ is a rectifying curve if and only if $$\rho(t) = \frac{a}{\cos(t+t_0)} \text{ with } a \in \mathbb{R}_0 \text{ and } t_0 \in \mathbb{R} .$$ 
\end{thr}
\begin{proof}
To make the notation less elaborate, we do not always write down explicitly the parameter $t$ in this proof.

The second order derivative of $\alpha(t)=\rho(t)y(t)$ with respect to $t$ is
\begin{equation} T'= \left(\frac{\rho'}{v } \right)'y+ \left(\frac{\rho'}{v } + \left( \frac{\rho}{v} \right)' \right) y' + \frac{\rho}{v}y''  \label{SphCurve}
\end{equation} 
where the speed $v$ of $\alpha$ is $v(t) = \sqrt{ \rho'(t)^2+\rho^2(t) }$ since $\langle y,y  \rangle = \langle y',y' \rangle=1$.

Take an orthonormal basis $\left\{y,y',Y_1, \ldots, Y_{n-2} \right\}$ of $T\mathbb{E}^n$.
Then, $$y''= \langle y,y''\rangle y+\langle y',y''\rangle y'+ \langle Y_1,y''\rangle Y_1+\cdots + \langle Y_{n-2},y'' \rangle Y_{n-2}.$$
Since also $ \langle y,y'' \rangle =-1$ and $ \langle y',y'' \rangle =0$, using the Equations \eqref{FrenetEqEn} for non-arclength parameterized curves, we can rewrite Equation \eqref{SphCurve} to
$$ \kappa_1 v N = \left(\left(\frac{\rho'}{v } \right)' - \frac{ \rho}{v} \right) y +  \left(\frac{\rho'}{v } + \left( \frac{\rho}{v} \right)' \right) y' +    \frac{\rho}{v } \left(\sum_{i=1}^{n-2}  \langle Y_i,y'' \rangle Y_i \right). $$
Now, take the scalar product of this last equation with $\alpha$ and use that $\alpha=\rho y$, then, 
$$ \kappa_1 v \langle \alpha,N\rangle =  \left( \left(\frac{\rho'}{v } \right)' - \frac{ \rho}{v} \right) \rho. $$
Since $\alpha$ is a rectifying curve if and only if $\langle \alpha, N\rangle=0$, this leads to the differential equation $$ \rho \rho''-2\rho'^2-\rho^2 = 0 $$ which has as non-zero solutions $\rho(t)= a\sec(t+t_0)$ for $a \in \mathbb{R}_0$ and $t_0 \in \mathbb{R}$.
\end{proof}

\begin{examp}
Theorem \ref{ThConstrRectCurveEn} allows us to give some explicit examples of rectifying curves in $\mathbb{E}^n$.

If $n$ is even, $n=2m$, then parameterization \eqref{CurveCstCurvEvenn} with $\sum _{i=1}^m a_i^2=1$ and $\sum _{i=1}^m a_i^2b_i^2=1$, describes an arclength parameterized curve in $\mathbb{S}^{n-1}(1)$.
Hence, $$ \alpha (t) = \frac{a}{\cos(t+t_0)} \left( a_1 \sin \left( b_1 t \right), a_1 \cos \left( b_1 t \right), \ldots , a_m \sin \left( b_m t \right), a_m \cos \left( b_m t \right) \right)$$
with $a \in \mathbb{R}_0$ and $t_0 \in \mathbb{R}$ parameterizes a rectifying curve in $\mathbb{E}^n$.

For odd $n=2m+1$, one has for instance, if $\sum _{i=1}^m a_i^2b_i^2=1$, the arclength parameterized curve 
$$ y(t) = \left( a_1 \sin \left( b_1 t \right), a_1 \cos \left( b_1 t \right), \ldots , a_m \sin \left( b_m t \right), a_m \cos \left( b_m t \right), c \right) $$
that lies on the unit hypersphere $\mathbb{S}^{n-1}(1)$ if $\sum _{i=1}^m a_i^2 +c^2=1$.
Here for $i\in \{1, 2, \ldots ,m \}$, one has $a_i \in \mathbb{R}$, $b_i \in \mathbb{R}$ all distinct numbers and $c \in \mathbb{R}$.
Then, $$ \alpha (t) = \frac{a}{\cos(t+t_0)} \left( a_1 \sin \left( b_1 t \right), a_1 \cos \left( b_1 t \right), \ldots , a_m \sin \left( b_m t \right), a_m \cos \left( b_m t \right), c \right) $$ with $a \in \mathbb{R}_0$ and $t_0 \in \mathbb{R}$ is the parameterization of a rectifying curve in $\mathbb{E}^n$.
\end{examp}


\section{Conclusions}
Where possible, we stated and proved a generalization of known results to rectifying curves in $\mathbb{E}^n$.
In order to do so, one needs to take into account the $n-1$ curvatures and $n$ Frenet vectors of a curve in $\mathbb{E}^n$.
Because of this, it is not possible to explicitly generalize all known results to rectifying curves in $\mathbb{E}^n$.

As is clear from for instance \cite{LucasOrtega}, it is interesting to translate and study the concept of rectifying curves to other spaceforms.
Also, studying curves in $\mathbb{E}^n$ for which the position vector always lies in the orthogonal complement of a binormal vector could be interesting.


\end{document}